\documentclass[11pt]{article}
\usepackage[top=2.5cm, bottom=2.5cm, left=2.5cm, right=2.5cm]{geometry}
\usepackage{graphicx}
\usepackage{makeidx}
\usepackage{graphics}
\graphicspath{ {images/} }
\usepackage{amsmath}
\usepackage{amsthm}
\usepackage{amsfonts}
\usepackage{amssymb}
\usepackage{epsfig}
\usepackage{tabularx}
\usepackage{latexsym}
\usepackage{mathtools}
\usepackage{epstopdf}
\usepackage{array}
\usepackage[utf8]{inputenc}

\newtheorem{lemma}{Lemma}
\newtheorem{result}{Result}

\newcommand{\nontwo}{non-2-colorable}

\newcommand{\twocolor}{2-colorable}
\newcommand{\twocoloring}{2-coloring}
\newcommand{\twocolorings}{2-colorings}
\newcommand{\mn}[1][n]{m(#1)}
\newcommand{\mninline}[1][n]{$m(#1)$}
\newcommand{\mvninline}[2]{$m_{#1}(#2)$}
\newcommand{\mnabtinline}[1][n]{$m_A(#1)$}
\newcommand{\mnabtmath}[1][n]{m_A(#1)}
\newcommand{\nuniform}[1][n]{$#1$-uniform}
\newcommand{\nontwocolornhg}[1][n]{\nontwo\ \nuniform[#1]\ \hg}
\newcommand{\nontwocolornhgs}[1][n]{\nontwocolornhg[#1]s}

\newcommand{\hg}{hypergraph}
\newcommand{\Hg}{Hypergraph}
\newcommand{\hgs}{hypergraphs}
\newcommand{\Hgs}{Hypergraphs}
\newcommand{\he}{hyperedge}

\newcommand{\hes}{hyperedges}

\newcommand{\hgnot}{$H = (V, E)$}
\newcommand{\hgprimenot}{$H' = (V', E')$}
\newcommand{\cushg}[1]{H_{#1} = (V_{#1}, E_{#1})}
\newcommand{\cushgnot}[1]{$\cushg{#1}$}
\newcommand{\cushgprime}[1]{H'_{#1} = (V'_{#1}, E'_{#1})}
\newcommand{\cushgprimenot}[1]{$\cushgprime{#1}$}
\newcommand{\mono}{monochromatic}

\newcommand{\monohe}{\mono\ \he}

\newcommand{\hyperedgecustset}[3]{$#1 = \{#2_1, #2_2, \ldots, #2_{#3}\}$}
\newcommand{\hyperedgeset}[2]{\hyperedgecustset{#1}{e}{#2}}
\newcommand{\vertexset}[3]{$#1 = \{#2_1, #2_2, \ldots, #2_{#3}\}$}
\newcommand{\vertexsubset}[3]{$#1 = \bigcup_{i \in #2} \{#3_i\}$}
\newcommand{\card}[1]{\lvert #1 \rvert}
\newcommand{\cardinline}[1]{$\card{#1}$}
\newcommand{\floor}[1]{\lfloor #1 \rfloor}
\newcommand{\ceil}[1]{\lceil #1 \rceil}
\newcommand{\bk}{best-known}
\newcommand{\colorinline}[1][]{$\chi_{#1}$}
\newcommand{\matchingv}[1][i]{$\{a_#1, b_#1\}$}

\newcommand{\mnlowerblocked}[2]{\bigg\lceil\frac{\binom{#1}{\lfloor #1/2
            \rfloor}} {{\binom{#1 - #2}{\lfloor #1/2 \rfloor - #2} + \binom{#1 -
            #2}{\lceil #1/2 \rceil - #2}}}\bigg\rceil}

\newcommand{\mnlowerschoheim}[3]{\Big\lceil{\frac{#1}{#2}
        \Big\lceil{\frac{#1 - 1}{#3}}\Big\rceil}\Big\rceil}

\newcommand{\mthirteengenabtimprov}{275835}
\newcommand{\meightmultiimprov}{1212}
\newcommand{\mthirteenmultiimprov}{203139}
\newcommand{\mforteenmultiimprov}{528218}
\newcommand{\msixteenmultiimprov}{3923706}
\newcommand{\mseventeenmultiimprov}{10375782}
\newcommand{\mthirteenblockimprov}{200889}
\newcommand{\msixteenblockimprov}{3308499}

\newcommand{\iitgaddress}{IIT Guwahati, India.}

\newcommand{\iitroparaddress}{IIT Ropar, India.}
\newcommand{\gmail}[1]{Email: \texttt{#1@gmail.com}}
\newcommand{\itt}[1]{\textit{#1}}

\title{Improved Bounds for Uniform Hypergraphs without Property B}
\author{
Sachin Aglave\footnote{\iitgaddress\ \gmail{sachin.aaglave}} 
\and
V. A. Amarnath\footnote{\iitgaddress\ \gmail{vaamarnath}}
\footnote{Corresponding Author}
\and
Saswata Shannigrahi\footnote{\iitroparaddress\ \gmail{saswata.shannigrahi}}
\and 
Shwetank Singh\footnote{\iitgaddress\ \gmail{shwetanksinghsrmscet}}
}

\begin{document}
\maketitle

\begin{abstract}
    A \hg\ is said to be properly \twocolor\ if there exists a \twocoloring\ of its
    vertices such that no \he\ is monochromatic. On the other hand, a \hg\ is
    called \nontwo\ if there exists at least one monochromatic \he\ in each of
    the possible 2-colorings of its vertex set.  Let \mninline\ denote the
    minimum number of \hes\ in a \nontwocolornhg.  Establishing the lower and
    upper bounds on \mninline\ is a well-studied research direction over
    several decades. In this paper, we present new constructions for non-$2$-colorable
    $n$-uniform hypergraphs. These constructions improve the upper bounds for $m(8)$,
    $m(13)$, $m(14)$, $m(16)$ and $m(17)$. We also improve the lower bound for \mninline[5]. 

    \vspace{.3cm}
\end{abstract}

\section{Introduction}

\Hgs\ are combinatorial structures that are generalizations of graphs.  Let
\hgnot\ be an \nuniform\ \hg\ with vertex set $V,$ with each \he\ in $E$ having
exactly $n$ vertices in it.  A \itt{2-coloring} of $H$ is an assignment of one
of the two colors red and blue to each of the vertices in $V$. We say a
\twocoloring\ of $H$ to be \itt{proper} if each of its \hes\ has red as well as
blue vertices. $H$ is said to be \itt{non-2-colorable} if no proper 2-coloring
exists for it; otherwise, it is said to satisfy \itt{Property B}. For an integer
$n \geq 1$, let \mninline\ denote the minimum number of \hes\ present in a
\nontwocolornhg.

Establishing an upper bound on \mninline\ is a well-explored combinatorial
problem.  Erd\H{o}s \cite{erdos1963combinatorial} gave a non-constructive proof
to establish the currently best-known upper bound $\mn = O(n^2 2^n)$. 
However, there is no known construction for a non-$2$-colorable $n$-uniform hypergraph
that matches this upper bound.
Abbott and Moser \cite{abbott1964combinatorial} constructed a non-$2$-colorable $n$-uniform hypergraph
with $O((\sqrt{7} + o(1))^n)$ hyperedges. Recently, Gebauer
\cite{gebauer2013construction} improved this result by constructing a non-$2$-colorable $n$-uniform hypergraph with $O(2^{(1 + o(1))n})$
hyperedges. Even though this is the best construction known for a
non-$2$-colorable $n$-uniform hypergraph for large $n$, it is still asymptotically far from the
above-mentioned non-constructive upper bound given by Erd\H{o}s.

Finding upper bounds for small values of $n$ is also a well-studied problem and
several constructions have been given for establishing these. For example, it
can be easily seen that $m(1) \leq 1$, $m(2) \leq 3$ (the corresponding $2$-uniform 
hypergraph is the triangle graph) and $m(3) \leq 7$ (the corresponding $3$-uniform 
hypergraph is known as the Fano plane \cite{klein1870theorie}, denoted by $H_f$ in this paper). The previously-mentioned 
construction of Abbott and Moser shows that $m(4) \leq 27$, $m(6) \leq 147$ and 
$m(8) \leq 2187$. Moreover, their construction also gives non-trivial upper
bounds on $m(n)$ for $n= 9, 10, 12, 14, 15$ and $16$. For $n \geq 3$, Abbott and Hanson
\cite{abbott1969combinatorial} gave a construction using a \nontwocolornhg[(n-2)]\ 
to show that $m(n) \leq n \cdot m(n - 2) + 2^{n - 1} + 2^{n-2}((n - 1) \mod 2)$.
Using the \bk\ upper bounds on \mninline[n-2], this recurrence relation establishes non-trivial
upper bounds as well as improve such bounds on \mninline\ for a few small values of $n$. For example, it shows
that $m(4) \leq 24$, $m(5) \leq 51$ and $m(7) \leq 421$. Seymour
\cite{seymour1974note} further improved the upper bound on \mninline[4]\ to $m(4) \leq
23$ by constructing a \nontwocolornhg[4]\ with 23 \hes. In this paper, we denote this hypergraph by $H_s$. 
For even integers $n \geq 4$, Toft \cite{toft1973color} generalized this construction using a
\nontwocolornhg[(n-2)]\ to improve Abbott and Hanson's result to $m(n) \leq n
\cdot m(n - 2) + 2^{n - 1} + \binom{n}{n/2}/2$. In particular, this led to
establishing an upper bound $m(8) \leq 1339$.  For a given integer $n \geq 3$
and a \nontwocolornhg[(n-2)]\ $A$, we refer to Abbott-Hanson's construction for
odd $n$ and Toft's construction for even $n$ as \itt{Abbott-Hanson-Toft
construction} and denote the number of \hes\ in such a \hg\ as \mnabtinline.
It can be easily observed that $\mn \leq \mnabtmath$ for any
\nontwocolornhg[(n-2)]\ $A$. In fact, we have already seen that the
above-mentioned upper bounds $\mn[4] \leq 23$, $\mn[5] \leq 51$, $\mn[7] \leq
421$ and $\mn[8] \leq 1339$ are obtained by Abbott-Hanson-Toft constructions
using the \bk\ constructions for \nontwo\ 2, 3, 5 and 6-uniform \hgs, respectively.  Recently, a
construction given by Mathews et al.  \cite{mathews2015construction} improved
the upper bound on $m(8)$ to $m(8) \leq 1269$.  In addition, they modified the
Abbott-Hanson-Toft construction to improve the upper bounds on \mninline\ for
$n = 11, 13$ and $17$. The currently \bk\ upper bounds on \mninline\ for $n \leq 17$ 
are given in Table \ref{table-cur-known-upper}.

\begin{table}
\begin{center}
\begin{tabular}{|c|c|c|}
\hline
$n$ & $m(n)$ & Corresponding construction/recurrence relation \\
\hline
1   & $m(1) = 1$            & Single Vertex\\
2   & $m(2) = 3$            & Triangle Graph\\
3   & $m(3) = 7$            & Fano Plane \cite{klein1870theorie}\\
4   & $m(4) = 23$           & \cite{ostergaard2014minimum},
\cite{seymour1974note}\\
5   & $m(5) \leq 51$        & $m(5) \leq 2^4 + 5 m(3)$\\
6   & $m(6) \leq 147$       & $m(6) \leq m(2) m(3)^2$\\
7   & $m(7) \leq 421$       & $m(7) \leq 2^6 + 7 m(5)$\\
8   & $m(8) \leq 1269$      & \cite{mathews2015construction}\\
9   & $m(9) \leq 2401$      & $m(9) \leq m(3)^4$\\
10  & $m(10) \leq 7803$     & $m(10) \leq m(2) m(5)^2$\\
11  & $m(11) \leq 25449$    & $m(11) \leq 15 \cdot 2^8 + 9 m(9)$\\
12  & $m(12) \leq 55223$    & $m(12) \leq m(3)^4 m(4)$\\
13  & $m(13) \leq 297347$   & $m(13) \leq 17 \cdot 2^{10} + 11 m(11)$ \\
14  & $m(14) \leq 531723$   & $m(14) \leq m(2) m(7)^2$\\
15  & $m(15) \leq 857157$   & $m(15) \leq m(3)^5 m(5)$\\
16  & $m(16) \leq 4831083$  & $m(16) \leq m(2) m(8)^2$\\
17  & $m(17) \leq 13201419$ & $m(17) \leq 21 \cdot 2^{14} + 15 m(15)$ \\
\hline
\end{tabular}
\end{center}
\caption{Best-known upper bounds on $m(n)$ for small values of $n$}
\label{table-cur-known-upper}
\end{table}

In the other direction, Erd\H{o}s \cite{erdos1963combinatorial} showed the lower
bound on \mninline\ to be $\mn = \Omega(2^n)$, which was later improved by Beck
\cite{beck19783} to $m(n) = \Omega(n^{1/3 - o(1)} 2^n)$. The currently \bk\
lower bound $m(n) = \Omega(\sqrt{\frac{n}{\ln n}} 2^n)$ was given by
Radhakrishnan and Srinivasan \cite{radhakrishnan2000improved}. Recently, a
simpler proof for the same result has been given by Cherkashin and Kozik
\cite{cherkashin2014note}. Note that there is a significant
asymptotic gap between the currently best-known lower and upper bounds on $m(n)$.
Even for small values of $n$, we are only aware
of a few lower bounds for \mninline\ that match the corresponding upper
bounds. It can be easily seen that $m(1) \geq 1$, $m(2) \geq 3$ and $m(3) \geq
7$ and therefore $m(1) = 1$, $m(2) = 3$ and $m(3) = 7$. Recently,
{\"O}sterg{\aa}rd \cite{ostergaard2014minimum} showed that $m(4) \geq 23$ and
established $m(4) = 23$ as a result.  
The exact values of \mninline\ are not yet known for $n \geq 5$, even though
it can be easily observed that $m(n+1) \geq m(n)$ for any $n \geq 1$. 

\subsection{Our Contributions}
In this paper, we give constructions that improve the \bk\ upper bounds on
$m(8)$, $m(13)$, $m(14)$, $m(16)$ and $m(17)$. We also establish a non-trivial
lower bound on \mninline[5]. 

In Section \ref{upper-const-one}, we give a construction that gives the
following recurrence relation. In particular, it improves the 
upper bound on $m(13)$. 
\begin{result}
    \label{result-gen-abt} 
    Consider an integer $k \geq 1$. 
    For an odd $n > 2k$,  
    $m(n) \leq \binom{n+k-1}{k} m(n-2k) + \binom{n+k-1}{k-1} 2^{n-1}$. 
    For an even $n > 2k$, 
    $m(n) \leq \binom{n+k-1}{k} m(n-2k) + \binom{n+k-1}{k-1} (2^{n-1} +
    \binom {n}{n/2}/2)$. 
\end{result}

This construction also gives a \nontwocolornhg\ with $O(3.76^n)$ \hes. Even
though we note that it gives a better constructive upper bound $m(n) =
O(3.76^n)$ than the trivial bound $m(n) \leq \binom{2n-1}{n} =
\Theta(4^{n}/\sqrt{n})$, it is asymptotically worse than the previously
mentioned constructive upper bounds  $m(n) = O((\sqrt{7} + o(1))^n)$
\cite{abbott1964combinatorial} and $m(n) = O(2^{(1 + o(1))n})$
\cite{gebauer2013construction}. 

In Section \ref{upper-const-two}, we provide another construction that improves
the upper bounds on $m(8)$, $m(13)$, $m(14)$, $m(16)$ and $m(17)$. 
\begin{result}
Consider an integer $k$ satisfying $0 < k < n $. Let $w = \floor{n/k}$, $x =
n\mod k$, $y = \floor{k/x}$ and $z = k\mod x$.
    \label{result-multicore}
    \begin{enumerate}
        \item[(a)] If $x > 0$ and $z > 0$, $m(n) \leq w \cdot m(n-k) m(k) + y \cdot
            m(k)^w m(x) + \tbinom{x+z-1}{z} m(n-k) m(x)^y + \tbinom{x+z-1}{x}
        m(k)^w$.
        \item[(b)] If $x > 0$ and $z = 0$, $m(n) \leq w \cdot m(n-k) m(k) + y \cdot
            m(k)^w m(x) + m(n-k) m(x)^y$.
        \item[(c)] If $x = 0$, $m(n) \leq w \cdot m(n-k) m(k) + m(k)^w$.
    \end{enumerate}
\end{result}

In Section \ref{upper-const-five}, we give a construction to prove the following
result that further improves the upper bounds on $m(13)$ and $m(16)$.
\begin{result}
    \label{result-block}
    Consider an integer $k \geq 2$ and a \nontwocolornhg[(k-1)]\ $H_{1c}$. 
    Then, $m(3k+1) \leq (m(k - 1) + 2^{k - 1})
    m(k + 1)^2 + 2 m_{H_{1c}}(k + 1) m(k)^2 + 4 m(k + 1) m(k)^2$.
\end{result}

\noindent In Section \ref{our-work-low}, we improve the currently \bk\ lower
bound $m(5) \geq 28$. 

\begin{result}
    \label{result-improv-lower}
    $m(5) \geq 29$.
\end{result}

\section{Previous Results \label{prev-works}} 

\subsection{Abbott-Moser Construction} 
\label{prev-result-abbott-moser}
Abbott and Moser \cite{abbott1964combinatorial} gave the construction for a
\nontwocolornhg\ \hgnot\ by exploiting the known constructions of \nontwo\
\nuniform[a]\  and \nuniform[b]\ \hgs\ for any composite $n$ satisfying $n = ab$
for two integers $a \geq 1, b \geq 1$.\footnote{Note that the notations used in a
    sub-section are not related to the notations used in other sub-sections,
    unless specified otherwise.} Let \cushgnot{a} and \cushgnot{b} be  \nontwo\
\nuniform[a]\ and \nuniform[b] \hgs, respectively. $H$ is constructed using
\cardinline{V_a} identical copies of $H_b$ by replacing each vertex of $H_a$ with
a copy of $H_b$.  Let us denote the copies of $H_b$ as $\cushg{b_1},
\cushg{b_2}, \ldots, \cushg{b_{\card{V_a}}}$. The vertex set of $H$ is $V =
V_{b_1} \cup V_{b_2} \cup \cdots \cup V_{b_{\card{V_a}}}$. The \he\ set of $H$ is
constructed as follows. For each \he\ $\{v_1, \ldots, v_a\}$ in $E_a$, the
following collection of \hes\ $\{\{e_1  \cup \cdots \cup e_a \}: e_1 \in
E_{b_{v_1}}, \ldots, e_a \in E_{b_{v_a}} \}$ is added to $E$. The resulting \hg\
$H$ is \nuniform\ and it is evident from the construction that it has
$\card{E_a} \card{E_b}^a$ \hes. Abbott and Moser \cite{abbott1964combinatorial} 
showed that $H$ is \nontwo, thereby proving the following result.

\begin{lemma}
    \label{prev-result-abbott-moser-lemma}
        For any composite $n$ satisfying $n = ab$ for integers $a, b \geq 1$, $m(n)
        \leq m(a) m(b)^a$.
\end{lemma}

This construction gives the \bk\ upper bounds for some small values of $n$. For
example, it shows that $m(6) \leq 147$, $m(9) \leq 2401$, $m(10) \leq 7803$,
$m(12) \leq 55223$, $m(14) \leq 531723$, $m(15) \leq 857157$ and $m(16) \leq 4831083$.

\subsection{Abbott-Hanson-Toft Construction}
\label{prev-result-abbott-hanson-toft}
As mentioned in the introduction, Abbott-Hanson's construction
\cite{abbott1969combinatorial} for odd $n$ along with Toft's construction
\cite{toft1973color} for even $n$ is referred to as Abbott-Hanson-Toft
construction. For a given $n \geq 3$, this construction is built using a
\nontwocolornhg[(n-2)], which we call as the \itt{core hypergraph} and denote by
\cushgnot{c}. Let its \he\ set be \hyperedgeset{E_c}{m_c}. 

Let $A$ and $B$ be two disjoint sets of vertices where \vertexset{A}{a}{n} and
\vertexset{B}{b}{n}, each disjoint with $V_c$. For a given $K \subset \{1, 2,
\ldots, n\}$, we define \vertexsubset{A_K}{K}{a}, \vertexsubset{B_K}{K}{b},
$\overline{A}_{K} = A \setminus A_{K}$ and $\overline{B}_{K} = B \setminus B_K$.

The construction of the \nontwocolornhg\ \hgnot\ is as follows. The vertex set
is $V = V_c \cup A \cup B$ and the \he\ set $E$ consists of the following
\hes: 
\begin{enumerate}
    \item[(i)] $e_i \cup \{a_j\} \cup \{b_j\}$ for every pair $i, j$ satisfying $1 \leq i
        \leq m_c$ and $1 \leq j \leq n$
    \item[(ii)] $A_K \cup \overline{B}_K$ for each $K$ such that \cardinline{K} is
        odd and $1 \leq \card{K} \leq \floor{n/2}$ 
    \item[(iii)] $\overline{A}_K \cup B_K$ for each $K$ such that \cardinline{K} is even
        and $2 \leq \card{K} \leq \floor{n/2}$
    \item[(iv)] $A$
\end{enumerate}

It is easy to observe that the number of \hes\ in $H$ is $2^{n-1} + n m_c$ for
odd $n$ and $2^{n-1} + n m_c + \binom{n}{n/2}/2$ for even $n$.  Abbott-Hanson
\cite{abbott1969combinatorial} and Toft \cite{toft1973color} proved that $H$ is
\nontwo, and the construction gives the upper bound on \mninline\ as follows.

\begin{lemma}
    \label{prev-result-abbott-hanson-toft-lemma}
    \[
        m(n) \leq 
        \begin{cases}
            2^{n-1} + n \cdot m(n-2) & \text{if } n \text{ is odd} \\
            2^{n-1} + n \cdot m(n -2) + \binom{n}{n/2}/2 & \text{if } n \text{
                is even}
        \end{cases}
    \]
\end{lemma}

Lemma \ref{prev-result-abbott-hanson-toft-lemma} gives the \bk\ upper bounds on
\mninline\ for $n=5$ and $7$ as $m(5) \leq 51$ and $m(7) \leq 421$, respectively.

\subsection{Mathews-Panda-Shannigrahi Construction}

The following construction of a \nontwocolornhg\ for $n \geq 3$ is an
improvement over the Abbott-Hanson-Toft construction mentioned above. Similar to the Abbott-Hanson-Toft
construction, this construction also utilizes a \nontwocolornhg[(n-2)]
\cushgnot{c} that is called the core \hg\ in Section
\ref{prev-result-abbott-hanson-toft}.  Let \hyperedgecustset{E_c}{e}{m_c}.  In
addition, this construction uses two disjoint vertex sets \vertexset{A}{a}{n}
and \vertexset{B}{b}{n}, each disjoint from $V_c$.  Let $B^1$ denote the ordered
set $B^1 = (b_1, b_2, \ldots, b_n)$, where the ordering is defined as $b_1 \prec
b_2 \prec \ldots \prec b_n$.  For any $1 \leq p \leq n$, let $B^p$ denote the
ordered set where $b_1$ and $b_p$ are swapped in the ordering. Let the ordered
set $B^p = (b_p, b_2, \ldots, b_{p-1}, b_1, b_{p+1}, \ldots, b_n)$ be denoted by
$(w^p_1, w^p_2, \ldots, w^p_n)$, where the ordering is given as $w^p_1 \prec
w^p_2 \prec \ldots \prec w^p_n$. For $K \subset \{1, 2, \ldots, n\}$, let
\vertexsubset{A_K}{K}{a}, $\overline{A}_K = A \setminus A_K$,
\vertexsubset{B_K^p}{K}{w^p} and $\overline{B}_K^p = B \setminus B_K^p$.

The construction of the \nontwocolornhg\ \hgnot\ is defined as follows. The
vertex set is $V = V_c \cup A \cup B$ and the \he\ set $E$ consists of following
\hes:
\begin{enumerate}
    \item[(i)] $e_i \cup \{a_j\} \cup \{b_j\}$ for every pair $i, j$ satisfying $1 \leq i \leq
        m_c$ and $2 \leq j \leq n$
    \item[(ii)] $A_K \cup \overline{B}_K^p$ for each $p$ satisfying $1 \leq p \leq n$, and
        each $K$ such that \cardinline{K} is odd and $1 \leq \card{K} \leq
        \floor{n/2}$
    \item[(iii)] $\overline{A}_K \cup B_K^p$ for each $p$ satisfying $1 \leq p \leq n$, and 
        each $K$ such that \cardinline{K} is even and $2 \leq \card{K} \leq
        \floor{n/2}$
    \item[(iv)] $A$
\end{enumerate}

It can be seen that the number of \hes\ in $H$ is at most $(n + 1) 2^{n-2} + (n - 1) m_c$ 
when $n$ is odd and $(n + 1)2^{n-2} + \binom{n}{n/2}/2
+ (n - 1) \Big(m_c + \binom{n-2}{(n-2)/2}\Big)$ when $n$ is even. Mathews et al.
\cite{mathews2015construction} showed that $H$ is \nontwo, which gives the
following result.

\begin{lemma}
    \[
        m(n) \leq 
        \begin{cases}
            (n + 1)2^{n-2} + (n - 1) \cdot m(n - 2) & \text{if } n \text{ is
                odd} \\
            (n + 1)2^{n-2} + \binom{n}{n/2}/2 + (n - 1) \Big(m(n - 2) +
            \binom{n-2}{(n-2)/2}\Big) & \text{if } n \text{ is even}
        \end{cases}
    \]
\end{lemma}

This result improved the upper bounds on \mninline[13]\ and \mninline[17]\ to
$m(13) \leq 357892$ and $m(17) \leq 14304336$, respectively. However, Mathews et al.
modified the above construction in the same paper to provide another construction that 
gives the following result. 

\begin{lemma}
  \[
        m(n) \leq 
        \begin{cases}
            (n + 4)2^{n-3} + (n - 2) \cdot m(n - 2) & \text{if } n \text{ is
                odd} \\
            (n + 4)2^{n-3} + (n - 2) \cdot m(n - 2) + n \binom{n-2}{(n-2)/2}/2 +
            \binom{n}{n/2}/2 & \text{if } n \text{ is even}
        \end{cases}
    \]
\end{lemma}

This construction improved the upper bound on \mninline[11]\ to $m(11) \leq
25449$ and further improved the above-mentioned upper bounds on \mninline[13]\ 
and \mninline[17]\ to $m(13) \leq 297347$ and $m(17) \leq 13201419$, respectively.

\section{Generalized Abbott-Hanson-Toft Construction \label{upper-const-one}} 

For any $k \geq 1$, we construct a \nontwocolornhg\ \hgnot\ for an integer $n$
satisfying $n > 2k$.  This construction uses a \nontwocolornhg[(n - 2k)]\
\cushgnot{c} with \hyperedgeset{E_c}{m_c}. Consider two disjoint sets of
vertices \vertexset{A}{a}{n+k-1} and \vertexset{B}{b}{n+k-1}, each disjoint with
$V_c$.  Let us define $\mathcal{I}_i$ to be the collection of all $i$-element
subsets of the set $\mathcal{N} = \{1, 2, \ldots, n + k - 1\}$. For an $I \in
\mathcal{I}_{k-1}$, consider $K_I \subseteq \mathcal{N} \setminus I$. We define
$A_I = \bigcup_{i \in I} \{a_i\}$, $B_I = \bigcup_{i \in I} \{b_i\}$, $A_{K_I} =
\bigcup_{i \in K_I} \{a_i\}$, $B_{K_I} = \bigcup_{i \in K_I}\{b_i\}$,
$\overline{A}_{K_I} = A \setminus (A_{K_I} \cup A_I)$ and $\overline{B}_{K_I} =
B \setminus (B_{K_I} \cup B_I)$.

The \nontwocolornhg\ \hgnot\ is constructed with the vertex set $V = V_c \cup A \cup
B$. The \he\ set $E$ consists of the following \hes:
\begin{enumerate}
    \item[(i)] $e_i \cup A_I \cup B_I$ for each $I \in \mathcal{I}_k$ and all $i$
        satisfying $1 \leq i \leq m_c$
    \item[(ii)] $A_{K_I} \cup \overline{B}_{K_I}$ for each $K_I$ such that
        \cardinline{K_I} is odd and $1 \leq \card{K_I} \leq \floor{n/2}$, for
        each $I \in \mathcal{I}_{k-1}$
    \item[(iii)] $\overline{A}_{K_I} \cup B_{K_I}$ for each $K_I$ such that
        \cardinline{K_I} is even and $0 \leq \card{K_I} \leq \floor{n/2}$, for
        each $I \in \mathcal{I}_{k-1}$
\end{enumerate}

The number of \hes\ in $H$ is $\binom{n+k-1}{k} m_c + \binom{n+k-1}{k-1}
2^{n-1}$ when $n$ is odd and $\binom{n+k-1}{k} m_c + \binom{n+k-1}{k-1}
\big( 2^{n-1} + \binom{n}{n/2}/2 \big)$ when $n$ is even. 

For a \twocoloring\ of $H$, we call \matchingv\ to be a \itt{matching pair} if
both the vertices are colored by the same color. For a given $I \in
\mathcal{I}_{k-1}$, we define $A_{{blue}_I}$ to be the set blue vertices in $A
\setminus A_I$, and $B_{{blue}_I} = \{b_i :  a_i \in A_{{blue}_I}\}$.  Let
$\overline{A}_{{blue}_I} = A \setminus (A_I \cup A_{{blue}_I})$ and
$\overline{B}_{{blue}_I} = \{b_i : a_i \in \overline{A}_{{blue}_I}\}$.

\begin{lemma}
    \label{lem-const-one-no-matching}
    Consider any $I \in \mathcal{I}_{k-1}$. If there is no matching pair of vertices
    between $A \setminus A_I$ and $B \setminus B_I$ in a \twocoloring\ \colorinline\ of
    \hg\ $H$, then there exists at least one \monohe\ in the coloring \colorinline.
\end{lemma}

\begin{proof}
    Assume for the sake of contradiction that \colorinline\ is a proper
    \twocoloring\ of \hg\ $H$ with no matching pair of vertices between $A
    \setminus A_I$ and $B \setminus B_I$.  We arrive at a contradiction in each
    of the cases below.
    \begin{enumerate}
        \item[Case $1.$] $1 \leq \card{A_{{blue}_I}} \leq \floor{n/2}$ 

            If \cardinline{A_{{blue}_I}} is odd, the \he\ $A_{{blue}_I} \cup
            \overline{B}_{{blue}_I}$ is \mono\ in blue.

            If \cardinline{A_{{blue}_I}} is even, the \he\
            $\overline{A}_{{blue}_I} \cup B_{{blue}_I}$ is \mono\ in red. 

        \item[Case $2.$] $\floor{n/2} < \card{A_{{blue}_I}} < n$

            If $n - \card{A_{{blue}_I}}$ is odd, the \he\
            $\overline{A}_{{blue}_I} \cup B_{{blue}_I}$ is \mono\ in red.

            If $n - \card{A_{{blue}_I}}$ is even, the \he\ $A_{{blue}_I} \cup
            \overline{B}_{{blue}_I}$ is \mono\ in blue. 

        \item[Case $3.$] $\card{A_{{blue}_I}} = 0$ or $\card{A_{{blue}_I}} = n$
            
            If $\card{A_{{blue}_I}} = 0$, $A \setminus A_I$ is \mono\ in red. If
            $\card{A_{{blue}_I}} = n$, $A \setminus A_I$ is \mono\ in blue.
    \end{enumerate}
\end{proof}

\begin{proof}[Proof of Result \ref{result-gen-abt}]
    \label{the-const-one-final}
    Let us assume for the sake of contradiction that there exists a proper
    \twocoloring\ \colorinline\ for \hg\ $H$. We know that the core \hg\ $H_c$
    is a \nontwocolornhg[(n-2k)] and thus has a \monohe\ in the coloring
    \colorinline.  Without loss of generality, assume $H_c$ to be \mono\ in red.
    The \hes\ added in Step (i) ensure that no more than $(k-1)$ matching pairs
    of red vertices exist in \colorinline. This implies that there exists
    an $I \in \mathcal{I}_{k-1}$ such that there is no matching pair of red
    vertices from $A' = A \setminus A_I$ and $B' = B \setminus B_I$. As a
    result, it follows from Lemma \ref{lem-const-one-no-matching} that there
    exists at least one matching pair of blue vertices from $A'$ and $B'$. Let
    \matchingv[p]\ be such a matching pair of blue vertices, where $a_p \in A'$
    and $b_p \in B'$.  This leads to a contradiction in each of the following
    cases.
    
    \begin{enumerate}
        \item[Case $1.$] $1 \leq \card{A_{{blue}_I}} \leq \ceil{n/2}$ 

            If $\card{A_{{blue}_I}}$ is odd, the \he\ $A_{{blue}_I} \cup
                \overline{B}_{{blue}_I}$ is \mono\ in blue.

            If $\card{A_{{blue}_I}}$ is even, the \he\ $\overline{B}_{{blue}_I}
            \cup \{b_p\} \cup A_{{blue}_I} \setminus \{a_p\}$ is \mono\ in blue. 

        \item[Case $2.$] $\ceil{n/2} < \card{A_{{blue}_I}} < n$

            If $n - \card{A_{{blue}_I}}$ is odd, $\card{\overline{B}_{blue_I} \cup
                \{b_p\}}$ is even. Therefore, the \he\ $\overline{B}_{{blue}_I} \cup
            \{b_p\} \cup A_{{blue}_I}  \setminus \{a_p\}$ is \mono\ in blue.

            If $n - \card{A_{{blue}_I}}$ is even, the \he\ $A_{{blue}_I} \cup
            \overline{B}_{{blue}_I}$ is \mono\ in blue. 

        \item[Case $3.$] $\card{A_{{blue}_I}} = n$
            
            If $\card{A_{{blue}_I}} = n$, $A'$ is \mono\ in blue.
    \end{enumerate}

    This completes the proof that $H$ is \nontwo. Therefore, we arrive at the
    following result.
    \[
        m(n) \leq 
        \begin{cases}
            \binom{n+k-1}{k} m(n-2k) + \binom{n+k-1}{k-1}2^{n-1} & \text{if } n
            \text{ is odd} \\
            \binom{n+k-1}{k} m(n-2k) + \binom{n+k-1}{k-1}\Big(2^{n-1} +
            \binom {n}{n/2}/2\Big) & \text{if } n \text{ is even}
        \end{cases}
    \]
\end{proof}

We set $k = 2$ in Result \ref{result-gen-abt} and use $m(9) \leq 2401$ from
Table \ref{table-cur-known-upper} to get an improvement of the upper bound on
\mninline[13]\ to $m(13) \leq \binom{14}{2} m(9) + \binom{14}{1} 2^{12} =
\mthirteengenabtimprov$.

\subsubsection{Optimization of \mninline}

From the construction above, we obtain the following for any integer $n$ greater than a
large constant $n_0 > 0$. 
\begin{align*}
    m(n) &\leq \binom{n+k-1}{k} m(n-2k) + \binom{n+k-1}{k-1}\bigg(2^{n-1} +
    \binom {n}{\lfloor n/2 \rfloor}/2\bigg) \\
    &\leq \binom{n+k-1}{k} m(n-2k) + \binom{n+k-1}{k-1} 2^n \\
    \shortintertext{Let $k = np$, where $\frac{1}{n} \leq p < 0.5$. Therefore,}
    m(n) &\leq \binom{n+np-1}{np} m(n-2np) + \binom{n+np-1}{np-1} 2^n \\
    &\leq \binom{n+np}{np} m(n-2np) + \binom{n+np}{np} 2^n. \\
    \shortintertext{Applying $\binom{n}{k} < (\frac{en}{k})^k$ and setting $n$ such
    that $n{(1-2p)}^i$ is an integer for all $i$ in the range $1 \leq i \leq s$, 
    we obtain the following.}
    m(n) &< \bigg(\frac{e(n+np)}{np}\bigg)^{np} \Big(m(n-2np) + 2^n\Big) \\
    &< \bigg(\frac{e(1+p)}{p}\bigg)^{np} \Bigg[2^n +
    \bigg(\frac{e(1+p)}{p}\bigg)^{(n-2np)p} \Bigg( m(n(1-2p)^2) + 2^{n-2np}\Bigg)
    \Bigg] \\ 
    &\hspace{2cm} \vdots \\
    &< m(n(1-2p)^s) \bigg(\frac{e(1+p)}{p}\bigg)^{\sum_{i=0}^{s-1} np(1-2p)^i}
    \\ & \hspace{1cm} +  \sum_{i=0}^{s-1} 2^{n(1-2p)^i}
    \bigg(\frac{e(1+p)}{p}\bigg)^{\sum_{j=0}^{i} np(1-2p)^j}\\
    &= m(n(1-2p)^s) \bigg(\frac{e(1+p)}{p}\bigg)^{\frac{n}{2}(1-(1-2p)^s)} \\ 
    & \hspace{1cm} + \sum_{i=0}^{s-1} 2^{n(1-2p)^i} 
    \bigg(\frac{e(1+p)}{p}\bigg)^{\frac{n}{2}(1-(1-2p)^{(i+1)})} \\
    &\leq m(n(1-2p)^s) \bigg(\frac{e(1+p)}{p}\bigg)^{\frac{n}{2}} \\
    & \hspace{1cm} + \bigg(\frac{e(1+p)}{p}\bigg)^{\frac{n}{2}} \sum_{i=0}^{s-1}
    \Bigg(\frac{2}{\Big(\frac{e(1+p)}{p}\Big)^{(1-2p)/2}}\Bigg)^{n(1-2p)^i}\\
    \shortintertext{For any integer $c > 0$, we observe that there exists a
        constant $c' > \frac{\ln n - \ln c}{2 \ln n}$ such that $n(1-2p)^s < c$ 
        for $s \geq c'n\ln n$. Using $s = c' n\ln n$ in the above equation, we have}
    m(n) &\leq \bigg(\frac{e(1+p)}{p}\bigg)^{\frac{n}{2}} \Bigg[ m(c) +
    \sum_{i=0}^{c'n\ln n - 1}
    \Bigg(\frac{2}{\Big(\frac{e(1+p)}{p}\Big)^{(1-2p)/2}}\Bigg)^{n(1-2p)^i} \Bigg]. \\
    \shortintertext{We observe that
        $2 \Big(\big(\frac{e(1+p)}{p}\big)^{(1-2p)/2}\Big)^{-1}$ increases if $p$
        increases and its value is less than 1 for $0 < p \leq 0.2381$. Using
        $p=0.2381$ in the above equation, we obtain}
    m(n) &< \Big(\frac{e(1+p)}{p}\Big)^{\frac{n}{2}} \big[ m(c) +
    c' n \ln n \big]\\
    &= O(3.7596^n \cdot n \ln n)\\
    &= O(3.76^n).
\end{align*}

\section{Multi-Core Construction \label{upper-const-two}}

Consider an integer $k$ satisfying $0 < k < n$. We define $w = \floor{n/k}, x = n \mod k, 
y = \floor{k/x}$ and $z = k \mod x$. A multi-core construction makes use of a
\nontwocolornhg[(n-k)]\ \cushgnot{c}, a total of $w$ identical \nontwocolornhgs[k]\
$\cushg{1}, \ldots, \cushg{w}$ and a total of $y$ identical \nontwocolornhgs[x]\
$\cushgprime{1}, \ldots, \cushgprime{y}$. The vertex sets of the \hgs\ $H_c$,
$H_1, \ldots, H_w$, $H'_1, \ldots, H'_y$ are pairwise disjoint. Let us denote
\hyperedgecustset{E_c}{e}{m_c}, \hyperedgecustset{E_1}{e^1}{m_k}, $\ldots$,
\hyperedgecustset{E_w}{e^w}{m_k}, \hyperedgecustset{E'_1}{e'^1}{m_x}, $\ldots$,
\hyperedgecustset{E'_y}{e'^y}{m_x}.  Consider a vertex set
\vertexset{A}{a}{x+z-1}, disjoint with each of $V_c$, $V_1, \ldots, V_w$, $V'_1,
\ldots, V'_y$. We define $\mathcal{A}_p$ as the collection of all $p$-element
subsets of the vertex set $A$. Let $\mathcal{E} = \{j_1 \cup j_2 \cup \cdots
\cup j_w : (j_1, j_2, \ldots, j_w) \in E_1 \times E_2 \times \cdots \times
E_w\}$ and $\mathcal{E'} = \{j'_1 \cup j'_2 \cup \cdots \cup j'_y : (j'_1, j'_2,
\ldots, j'_y) \in E'_1 \times E'_2 \times \cdots \times E'_y\}$. 

We define the construction of a \nontwocolornhg\ \hgnot\ as follows. The vertex
set is $V = V_c \cup A \cup V_1 \cup \cdots \cup V_w \cup V'_1 \cup \cdots \cup
V'_y$. The construction of the \hes\ belonging to $E$ depends on the values of
$x$ and $z$ as follows. 

\begin{enumerate}
    \item[Case $1.$] For $x > 0$ and $z > 0$, $E$ contains the following \hes.

    \begin{enumerate}
        \item[(i)] $e_i \cup e_j^l$ for every triple $i, j, l$ satisfying $1 \leq i \leq
            m_c$, $1 \leq j \leq m_k$ and $1 \leq l \leq w$ \label{mul-core-step-1}
        \item[(ii)] $e'^j_i \cup e$ for every triple $i, j, e$ satisfying $1 \leq i \leq
            m_x$, $1 \leq j \leq y$ and $e \in \mathcal{E}$ \label{mul-core-step-2}
        \item[(iii)] $e_i \cup e' \cup S $ for every triple $i, e, S$ satisfying $1 \leq i
            \leq m_c$, $e' \in \mathcal{E'}$ and $S \in \mathcal{A}_z$
            \label{mul-core-step-3}
        \item[(iv)] $e \cup S$ for every pair $e, S$ satisfying $e \in \mathcal{E}$ and $S
            \in \mathcal{A}_x$ \label{mul-core-step-4}
    \end{enumerate}

    \item[Case $2.$] For $x > 0$ and $z = 0$, $E$ contains the following \hes.

    \begin{enumerate}
        \item[(i)] $e_i \cup e_j^l$ for every triple $i, j, l$ satisfying $1 \leq i \leq
            m_c$, $1 \leq j \leq m_k$ and $1 \leq l \leq w$
        \item[(ii)] $e'^j_i \cup e$ for every triple $i, j, e$ satisfying $1 \leq i \leq
            m_x$, $1 \leq j \leq y$ and $e \in \mathcal{E}$
        \item[(iii)] $e_i \cup e'$ for every pair $i, e'$ satisfying $1 \leq i \leq m_c$
            and $e' \in \mathcal{E'}$
    \end{enumerate}

    \item[Case $3.$] For $x = 0$, $E$ contains the following \hes.

    \begin{enumerate}
        \item[(i)] $e_i \cup e_j^l$ for every triple $i, j, l$ satisfying $1 \leq i
            \leq m_c$, $1 \leq j \leq m_k$ and $1 \leq l \leq w$
        \item[(ii)] $e$ for each  $e \in \mathcal{E}$
    \end{enumerate}
\end{enumerate}

\noindent The number of \hes\ in $H$ is given by

\begin{align*}
    \card{E} &= \begin{cases}
        w m_c m_k + y m_x (m_k)^w + \binom{x + z - 1}{z} m_c (m_x)^y +
        \binom{x + z - 1}{x} (m_k)^w & \text{if } x>0, z>0 \\
        w m_c m_k + y m_x (m_k)^w + m_c (m_x)^y & \text{if } x>0, z=0 \\
        w m_c m_k + (m_k)^w  & \text{if } x=0
    \end{cases}
\end{align*}

\begin{proof}[Proof of Result \ref{result-multicore}] 
    \label{the-const-two-final}
    For the sake of contradiction, let us assume that \colorinline\ is a proper
    \twocoloring\ of $H$.  Without loss of generality, let the \hg\ $H_c$ be
    monochromatic in red in the coloring \colorinline.  The \hes\ formed in
    Step (i) in each of the cases ensure that the \hgs\ $H_j$ are monochromatic
    in blue for each $j \in  \{1, \ldots, w\}$.

    \begin{enumerate}
        \item[{Case} 1.] If $x > 0$ and $z > 0$, the \hes\ formed in Step (ii) ensure that
            the \hgs\ $H'_l$ are monochromatic in red for each $l \in \{1, 2,
            \ldots, y\}$.  It can be noted from the \hes\ generated in Step
            (iii) that there are at most $z-1$ red vertices in the
            set $A$. This implies that $A$ has at least $x$ blue vertices.
            The \hes\ formed in Step (iv) ensure that there are at most
            $x-1$ blue vertices in $A$. Thus, we have a contradiction.
        \item[{Case} 2.] If $x > 0$ and $z = 0$, the \hes\ formed in Step (ii) ensure that
            the \hgs\ $H'_l$ are monochromatic in red for each $l \in \{1, 2,
            \ldots, y\}$.  It can be easily noted that the \hes\ generated in
            Step (iii) include a red \monohe. Thus, we have a contradiction. 
        \item[{Case} 3.] If $x = 0$, it immediately follows that we have a blue \monohe\ in
            the \hes\ generated by Step (ii) of the construction.  This leads to
            a contradiction. 
    \end{enumerate}
 
    \noindent Thus, we have the following result on \mninline.
    \begin{align*}
        \shortintertext{If $x > 0$ and $z > 0$,}
        m(n) &\leq w \cdot m(n-k) m(k) + y \cdot m(k)^w m(x) \\ &
        \hspace{1cm} + \tbinom{x+z-1}{z} m(n-k) m(x)^y + \tbinom{x+z-1}{x}
        m(k)^w. \\
        \shortintertext{If $x > 0$ and $z = 0$,}
        m(n) &\leq w \cdot m(n-k) m(k) + y \cdot m(k)^w m(x) + m(n-k) m(x)^y. \\
        \shortintertext{If $x = 0$,}
        m(n) &\leq w \cdot m(n-k) m(k) + m(k)^w.
    \end{align*}
\end{proof} 

These recurrence relations give improvements on \mninline[n]\ for $n = 8, 13,
14, 16$ and $17$ as follows:
\begin{itemize}
    \item For $n = 8$ and $k = 5$, we have $m(8) \leq m(3) m(5) + m(5) m(3) +
        \tbinom{4}{2} m(3) m(3) + \tbinom{4}{3} m(5) \leq \meightmultiimprov$ by
        using $m(3) = 7$ and $m(5) \leq 51$ from Table \ref{table-cur-known-upper}. 
    \item For $n = 13$ and $k = 5$, we obtain $m(13) \leq 2 m(8) m(5) + m(5)^2
        m(3) + \tbinom{4}{2} m(8) m(3) + \tbinom{4}{3} m(5)^2 \leq
        \mthirteenmultiimprov$ by using $m(3) = 7$ and $m(5) \leq 51$ from Table
        \ref{table-cur-known-upper} and $m(8) \leq 1212$ obtained above.
    \item For $n = 14$ and $k = 5$, the recurrence relation gives $m(14) \leq 2 m(9) m(5)
        + m(5)^2 m(4) + \tbinom{4}{1} m(9) m(4) + \tbinom{4}{4} m(5)^2 \leq
        \mforteenmultiimprov$ by using $m(4) = 23$, $m(5) \leq 51$ and $m(9) \leq
        2401$ from Table \ref{table-cur-known-upper}. 
    \item For $n = 16$ and $k = 7$, we have $m(16) \leq 2 m(9) m(7) + 3 m(7)^2
        m(2) + \tbinom{2}{1} m(9) m(2)^3 + \tbinom{2}{2} m(7)^2 \leq
        \msixteenmultiimprov$ by using $m(2) = 3$, $m(7) \leq 421$ and $m(9)
        \leq 2401$ from Table \ref{table-cur-known-upper}. 
    \item Finally, for $n = 17$ and $k = 7$, we obtain $m(17) \leq 2 m(10) m(7)
        + 2 m(7)^2 m(3) + \tbinom{3}{1}m(10) m(3)^2 + \tbinom{3}{3} m(7)^2 \leq
        \mseventeenmultiimprov$ by using $m(3) = 7$, $m(7) \leq 421$ and $m(10)
        \leq 7803$ from Table \ref{table-cur-known-upper}. 
    \item It can also be noted that this construction matches the currently
        \bk\ upper bounds on $m(6)$ and $m(10)$ for $k = 3$ and $k =
        5$, respectively.
\end{itemize}

\section{Block Construction} 
\label{upper-const-five}

For an integer $k > 0$, we describe the construction of a collection
$\mathcal{H}$ of \nontwocolornhgs. Any \hg\ \hgnot\ belonging to this collection
is constructed using a \nontwocolornhg[(n-2k)]\ denoted by \cushgnot{c} and two
disjoint collections of \hgs\ $\mathcal{A}$ and $\mathcal{B}$. Let
\hyperedgecustset{E_c}{e}{m_c}.  Let \vertexset{\mathcal{A}}{H}{t} and
\vertexset{\mathcal{B}}{H'}{t} be the collection of \hgs\ such that each of
\cushgnot{i} and \cushgprimenot{i} is an identical copy of a \nontwo\
\nuniform[k_i]\ \hg\ satisfying $k_i \geq k$ and $\sum_{i=1}^{t} k_i \geq n$.
Note that the sets $V_c, V_1, V_2, \ldots, V_t, V'_1, V'_2, \ldots V'_t$
are pairwise disjoint.

Let $P = \{i_1, i_2, \ldots, i_p\} \subseteq \{1, 2, \ldots, t\}$ such that $1
\leq i_1 < i_2 < \ldots < i_p \leq t$. Using the Cartesian products
$\mathcal{C}_P = E_{i_1} \times E_{i_2} \times \cdots \times E_{i_p}$ and
$\mathcal{C}'_P = E'_{i_1} \times E'_{i_2} \times \cdots \times E'_{i_p}$, let
us define the collection of \hes\ $\mathcal{A}_P$ and $\mathcal{B}_P$ as
$\mathcal{A}_P = \{ j_1 \cup j_2 \cup \cdots \cup j_p : (j_1, j_2, \ldots, j_p)
\in \mathcal{C}_P\}$ and $\mathcal{B}_P =  \{ j'_1 \cup j'_2 \cup \cdots \cup
j'_p : (j'_1, j'_2, \ldots, j'_p) \in \mathcal{C}'_P\}$, respectively. Also, let
$\overline{P} = \{1, 2, \ldots, t\} \setminus P$.

The \hg\ $H$ has the vertex set $V = V_c \cup V_1 \cup \cdots \cup V_t \cup V'_1
\cup \cdots \cup V'_t$ and the \he\ set $E$ is generated from the following
\hes, each containing at least $n$ vertices.
\begin{enumerate}
    \item[(i)] For each $j$ satisfying $1 \leq j \leq t$, $e_i \cup e_{H_j} \cup
        e_{H'_j}$ for every triple $i, e_{H_j}, e_{H'_j}$ satisfying $1 \leq
        i \leq m_c$, $e_{H_j} \in E_j$ and $e_{H'_j} \in E'_j$ 
        \label{const-four-step1} 
    \item[(ii)] For each $P \subset \{1, 2, \ldots, t\}$ such that \cardinline{P} is
        odd and $1 \leq \card{P} \leq \floor{t/2}$,  $e_{H} \cup e_{H'}$ for
        every pair $e_H, e_{H'}$ satisfying $e_{H} \in \mathcal{A}_P$ and $e_{H'}
        \in \mathcal{B}_{\overline{P}}$ \label{const-four-step2}
    \item[(iii)] For each $P \subset \{1, 2, \ldots, t\}$ such that \cardinline{P} is
        even and $0 \leq \card{P} \leq \floor{t/2}$,  $e_{H} \cup e_{H'}$ for
        every pair $e_H, e_{H'}$ satisfying $e_{H} \in \mathcal{A}_{\overline{P}}$
        and $e_{H'} \in \mathcal{B}_P$ \label{const-four-step3}
\end{enumerate}

We select an arbitrary set of $n$ vertices from each of the \hes\ generated
above to form the \he\ set $E$. In case a \he\ is included more than once in $E$
by this process, we keep only one of those to ensure that $E$ is not a
multi-set.  Let us count the number of \hes\ added to the \he\ set $E$. Step
$(i)$ adds at most $\card{E_c} \sum_{i=1}^{t} \card{E_i}
\card{E'_i} = \sum_{i=1}^{t} \card{E_i}^2 \card{E_c}$ \hes, whereas Steps
$(ii)$ and $(iii)$ together add at most 
$\prod_{i=1}^{t} \card{E_i} \big(1 + \binom{t}{1} + \ldots +
\binom{t}{\floor{t/2}} \big)$ \hes. Note that $\card{E} \leq \sum_{i=1}^{t}
\card{E_i}^2 \card{E_c} + 2^{t-1} \prod_{i=1}^{t} \card{E_i}$ when $t$ is odd,
and $\card{E} \leq \sum_{i=1}^{t} \card{E_i}^2 \card{E_c} + \big(2^{t-1} +
\binom{t}{t/2}/2 \big) \prod_{i=1}^{t} \card{E_i}$ when $t$ is even. In the following lemma,
we prove that $H$ is \nontwo\ by showing that any proper 2-coloring of $H$
can be used to obtain a proper 2-coloring of any \nuniform[t]\ \hg\ constructed by Abbott-Hanson-Toft
construction.

\begin{lemma}
    $H$ is \nontwo.
    \label{upper-const-block-lemma}
\end{lemma}

\begin{proof}
    Consider any \nuniform[t]\ \hg\ \cushgnot{{AHT}} constructed by
    Abbott-Hanson-Toft construction using a \nontwo\ \nuniform[(t-2)]\ core \hg\
    and two disjoint vertex sets $\{p_1, \ldots, p_t\}$ and $\{q_1, \ldots, q_t\}$.
    Assuming for the sake of contradiction that a proper \twocoloring\ exists
    for $H$, we give a proper \twocoloring\ for $H_{AHT}$ as follows. 
    \begin{itemize}
        \item Color all vertices of the \nontwo\ \nuniform[(t-2)]\ core \hg\
            of $H_{AHT}$ with the color of the \monohe\ of $H_c$ used in the
            construction of $H$.
        \item Color each vertex $p_i$ with the color of the \monohe\ of $H_i$
            used in the construction of $H$.
        \item Similarly, color each vertex $q_i$ with the color of the \monohe\ of
            $H'_i$ used in the construction of $H$.
    \end{itemize}

    Since $H_{AHT}$ is \nontwo, we have a contradiction. As a result, we have
    the following recurrence relation.
    \begin{align*}
        m(n) \leq \begin{cases}
            m(n - 2k) \sum_{i=1}^{t} m(k_i)^2  + 2^{t-1} \prod_{i=1}^{t}
              m(k_i) & \text{ if } t \text{ is odd} \\
              m(n - 2k) \sum_{i=1}^{t} m(k_i)^2 + \big(2^{t-1} +
              \binom{t}{t/2}/2 \big) \prod_{i=1}^{t} m(k_i) & \text { if } t
              \text{ is even}
        \end{cases}
    \end{align*}
\end{proof}

Consider the special case when $n = 3k+1$. Setting the values of $t$ and $k_i$'s
as $t = 3$, $k_1 = k + 1$ and $k_2 = k_3 = k$ in this special case, we obtain
the following recurrence relation. 
\begin{align}
    m(3k+1) & \leq m(k+1)^3 + 6 m(k)^2 m(k+1) \label{eq-block-con-gen-3k1}
\end{align}

\noindent We give an improvement of this result below.

\subsubsection*{Modified Block Construction}
Let us first repeat the detailed description for the special case mentioned above,
i.e., the construction of a \nontwocolornhg[(3k+1)]\ \hgnot\ belonging to
$\mathcal{H}$. We construct $H$ using a \nontwocolornhg[(k+1)]\ \cushgnot{c}\ 
along with \nontwocolornhgs[(k+1)]\ \cushgnot{1}, \cushgprimenot{1} 
and \nontwocolornhgs[k]\ \cushgnot{2}, \cushgprimenot{2}, \cushgnot{3},
\cushgprimenot{3}. Note that each $H'_i$ is an identical copy of $H_i$ for $1
\leq i \leq 3$.

For the modified construction described below, we set $H_1$ as the
Abbott-Hanson-Toft construction that uses a \nontwo\ \nuniform[(k-1)] core \hg\
\cushgnot{{1c}} and disjoint vertex sets \vertexset{A}{a}{k+1}, \vertexset{B}{b}{k+1}.
Note that $H'_1$ is not necessarily identical to $H_1$ in this modified block
construction, whereas each $H'_i$ is an identical copy of $H_i$ for $2 \leq i
\leq 3$. 
  
Using the notations introduced above, the vertex set of $H$ is $V = V_c \cup V_{1c}
\cup A \cup B \cup V'_1 \cup V_2 \cup V'_2 \cup V_3 \cup V'_3$. The \he\ set $E$
is generated from the following \hes.

\begin{enumerate}
    \item[(a)] $e_{H_c} \cup e_{H_1} \cup e_{H'_1}$ for every triple $e_{H_c}, e_{H_1},
        e_{H'_1}$ satisfying $e_{H_c} \in E_c$, $e_{H_1} \in E_1$ and
        $e_{H'_1} \in E'_1$ 
    \item[(b)] $e_{H_c} \cup e_{H_2} \cup
        e_{H'_2}$ for every triple $e_{H_c}, e_{H_2}, e_{H'_2}$ satisfying 
        $e_{H_c} \in E_c$, $e_{H_2} \in E_2$ and $e_{H'_2} \in E'_2$ 
    \item[(c)] $e_{H_c} \cup e_{H_3} \cup
        e_{H'_3}$ for every triple $e_{H_c}, e_{H_3}, e_{H'_3}$ satisfying 
        $e_{H_c} \in E_c$, $e_{H_3} \in E_3$ and $e_{H'_3} \in E'_3$ 
    \item[(d)] $e_{H_1} \cup e_{H'}$ for every pair $e_{H_1}, e_{H'}$ satisfying
        $e_{H_1} \in E_1$ and $e_{H'} \in \{j'_2 \cup j'_3 : (j'_2, j'_3) \in
        E'_2 \times E'_3\}$
    \item[(e)] $e_{H_2} \cup e_{H'}$ for every pair $e_{H_2}, e_{H'}$ satisfying
        $e_{H_2} \in E_2$ and $e_{H'} \in \{j'_1 \cup j'_3 : (j'_1, j'_3) \in
        E'_1 \times E'_3\}$
    \item[(f)] $e_{H_3} \cup e_{H'}$ for every pair $e_{H_3}, e_{H'}$ satisfying
        $e_{H_3} \in E_3$ and $e_{H'} \in \{j'_1 \cup j'_2 : (j'_1, j'_2) \in
        E'_1 \times E'_2\}$
    \item[(g)] All elements of the set  $\{j_1 \cup j_2 \cup j_3 :
        (j_1, j_2, j_3) \in E_1 \times E_2 \times E_3\}$
\end{enumerate}

Note that each of the \hes\ formed in Steps (b) to (g) has $3k + 1$ vertices. However,
the \hes\ formed in Step (a) have $3k + 3$ vertices in each of them. We can remove
any two vertices from each of these \hes\ to obtain the following recurrence
relation. Recall that $m_{H_{1c}}(k+1)$ denotes the number of \hes\ in the
\nontwocolornhg[(k+1)]\ constructed by Abbott-Hanson-Toft construction that uses
$H_{1c}$ as its core \hg.
\begin{align}
    m(3k+1) \leq m_{H_{1c}}(k+1)m(k+1)^2 + 2m_{H_{1c}}(k+1)m(k)^2 + 4m(k+1)m(k)^2
    \label{eq-block-con-abt}
\end{align}
Whenever $m(k+1) < m_{H_{1c}}(k+1)$, it is evident that the upper bound on
\mninline[3k+1]\ that this recurrence relation gives is worse than the one given
by Eq. \ref{eq-block-con-gen-3k1}. However, we observe that we can improve Eq.
\ref{eq-block-con-abt} by carefully selecting the two vertices to be removed
from each \he\ formed in Step (a). Recall that each of these \hes\ is a union of
three \hes\ $e_{H_c} \in E_c$, $e_{H_1} \in E_1$ and $e_{H'_1} \in E'_1$.  In
the following paragraph, we describe a process to create a set of $k-1$ vertices
from each \he\ in the \nuniform[(k+1)]\ \hg\ \cushgnot{1}. For each \he\
$e_{H_c} \cup e_{H_1} \cup e_{H'_1}$ formed in Step (a), we use this process to
remove two vertices from $e_{H_1}$.

\noindent Given a \he\ $h \in E_1$, we create a set $h'$ containing $k-1$ vertices 
as follows. 
\begin{enumerate}
    \item[{Case} 1.] If $h$ is created by Step (i) of Abbott-Hanson-Toft construction,
        i.e., if $h$ is of the form $e \cup \{a_i\} \cup \{b_i\}$ for some $e \in
        E_{1c}$, $a_i \in A$ and $b_i \in B$, we define $h' = e$. In other
        words, we remove $a_i$ and $b_i$ from $h$ to create $h'$.
    \item[{Case} 2.] If $h$ is created in Step (ii) of Abbott-Hanson-Toft construction,
        i.e., if $h$ is of the form $A_K \cup \overline{B}_{K}$ for some $K \subset 
        \{1, \ldots, k + 1\}$ such that \cardinline{K} is odd and $1 \leq \card{K}
        \leq \floor{(k+1)/2}$, we define $h' = A_K \cup \overline{B}_{K}
        \setminus \{a_{k}, a_{k + 1}, b_k, b_{k + 1}\}$.
    \item[{Case} 3.] If $h$ is created in Step (iii) of Abbott-Hanson-Toft construction,
        i.e., if $h$ is of the form $\overline{A}_K \cup B_{K}$ for some $K \subset 
        \{1, \ldots, k + 1\}$ such that \cardinline{K} is even and $2 \leq \card{K}
        \leq \floor{(k+1)/2}$, we define $h' = \overline{A}_K \cup B_{K}
        \setminus \{a_{k}, a_{k + 1}, b_k, b_{k + 1}\}$.
    \item[{Case} 4.] If $h$ is formed in Step (iv) of Abbott-Hanson-Toft construction,
        i.e., if $h = A$, we define $h' = A \setminus \{a_k, a_{k+1}\}$.
\end{enumerate}

\noindent This completes the construction of the \nuniform[(3k+1)]\ \hg\ $H$.

\begin{proof}[Proof of Result \ref{result-block}]
    We improve the recurrence relation given in Eq. \ref{eq-block-con-abt} as a
    result of selecting $k-1$ vertices from each $h \in E_1$, as described
    above. Since this process generates multiple copies of some $(k-1)$-element
    vertex sets,  the number of distinct \hes\ formed in Step (a) in the
    construction of $H$ is reduced. Let us determine the cardinality of the set
    $\{h' : h' \text{ is generated from some } h \in E_1\}$.

    It is easy to observe that the number of distinct $h'$'s formed in Case 1 is
    $\card{E_{1c}}$.  On the other hand, the total number of distinct $h'$'s
    formed in Cases 2, 3 and 4 is at most $2^{k-1}$. It follows from the fact
    that there are $2^{k-1}$ subsets of $A \setminus \{a_k, a_{k+1}\}$ and each
    $h'$ formed in one of the Cases 2, 3 and 4 is a union of the sets
    $\bigcup_{i \in P} \{a_i\}$ and $\bigcup_{i
        \in \{1, \ldots, k-1\} \setminus P} \{b_i\}$ for some $P \subseteq \{1, \ldots, k-1\}$.

    Since we have shown in Lemma \ref{upper-const-block-lemma} that $H$ is
    \nontwo, we have the following improvement over Eq. \ref{eq-block-con-abt}. 
    \begin{align*}
        m(3k+1) &\leq (m(k-1) + 2^{k-1}) m(k+1)^2 \\& \hspace{1cm} + 
        2m_{H_{1c}}(k+1) m(k)^2 + 4m(k+1) m(k)^2 
    \end{align*}
\end{proof}

\noindent This result improves the upper bounds on \mninline\ for $n = 13$ and 16 as
follows.
\begin{itemize}
    \item For $n = 13$, we have $k = 4$. Note that $m_{H_{1c}}(5) = 51$, when the
        Fano plane \cite{klein1870theorie} $H_f$ having 7 \hes\ is used as the core \hg\
        $H_{1c}$. Therefore, we obtain $m(13) \leq (m(3) + 2^3)m(5)^2 +
        2m_{H_{1c}}(5)m(4)^2 + 4m(5)m(4)^2 \leq \mthirteenblockimprov$ by using
        $m(3) = 7$, $m(4) = 23$ and $m(5) \leq 51$ from Table
        \ref{table-cur-known-upper}. 
    \item For $n = 16$, we have $k = 5$. Note that $m_{H_{1c}}(6) = 180$, when the 
        \nontwocolornhg[4]\ $H_s$ with 23 \hes\ is used as the core \hg\ $H_{1c}$.
        Therefore, we obtain $m(16) \leq (m(4) + 2^{4})m(6)^2 +
        2m_{H_{1c}}(6)m(5)^2 + 4m(6)m(5)^2 \leq \msixteenblockimprov$ by using
        $m(4) = 23$, $m(5) \leq 51$ and $m(6) \leq 147$ from Table
        \ref{table-cur-known-upper}. 
\end{itemize}

\section{Improved Lower Bound for $m(5)$}
\label{our-work-low}

For the sake of completeness, we begin this section with a proof of the result
given by Goldberg and Russell \cite{Goldberg93towardcomputing} for the lower
bounds on \mninline\ for small values of $n$. This result uses Lemma
\ref{lem:lower_erdos} and Lemma \ref{lem:lower_schonheim} in its proof.  Let
\mvninline{l}{n} be the minimum number of \hes\ in a \nontwocolornhg\ with $l$
vertices. 

\begin{lemma}\cite{erdos1969combinatorial} $m_{2n-1}(n) = m_{2n}(n) = \tbinom{2n
        - 1}{n}$. \label{lem:lower_erdos}
\end{lemma}

\begin{lemma}\cite{colbourn2006handbook} (Sch\"onheim bound)   
    Consider positive integers $l \geq n \geq t \geq 1$ and $\lambda \geq 1$.
    Any \nuniform\ \hg\ with $l$ vertices such that every $t$-subset of its
    vertices is contained in at least $\lambda$ \hes\ has at least $\Big\lceil
    \frac{l}{n} \Big\lceil \frac{l - 1}{n - 1} \cdots \Big\lceil \frac{\lambda
        (l - t + 1)}{n - t + 1}\Big\rceil \cdots \Big\rceil \Big\rceil $ \hes.
    \label{lem:lower_schonheim}
\end{lemma}

\begin{lemma} \cite{Goldberg93towardcomputing}
    If $n \geq 4$, then $\displaystyle m(n) \geq \min_{x > 2n, x \in \mathbb{N}}
    \bigg\{ \max \bigg\{\mnlowerblocked{x}{n}, \mnlowerschoheim{x}{n}{n-1}
    \bigg\} \bigg\}$. \label{lem:goldberg}
\end{lemma}

\begin{proof}
    Let us consider an \nuniform\ \hg\ \hgnot\ such that the number of \hes\
    satisfies $\displaystyle \card{E} < \min_{x > 2n, x \in \mathbb{N}} \bigg\{
    \max \bigg\{ \mnlowerblocked{x}{n}, \mnlowerschoheim{x}{n}{n-1} \bigg\}
    \bigg\}$.  We call a \twocoloring\ of the \hg\ to be \itt{balanced} if the
    coloring has $\floor{\card{V}/2}$ red vertices and $\ceil{\card{V}/2}$ blue
    vertices. It can be noted that the possible number of ways to give a
    balanced coloring for $H$ is $\binom{\card{V}}{\floor{\card{V}/2}}$ and not
    all of these are proper \twocolorings. Let us define $f(x) =
    \mnlowerblocked{x}{n}$, $g(x) = \mnlowerschoheim{x}{n}{n-1}$ and
    $\displaystyle r = \min_{x > 2n, x \in \mathbb{N}} \big\{ \max \big\{ f(x),
    g(x) \big\} \big\}$. Let this minimum value $r$ be obtained by $x =
    v_{\text{opt}}$.  When $x > 2n$, observe that $f(x)$ is non-increasing and
    $g(x)$ is non-decreasing with increasing $x \in \mathbb{N}$. Moreover, we
    also observe that $\binom{2n-1}{n} \geq f(2n+1) \geq g(2n+1)$ for $n \geq
    4$.

    \begin{enumerate}
        \item[{Case }1.] If $n \leq \card{V} \leq 2n - 2$, any balanced coloring of its vertex
            set is a proper \twocoloring\ of $H$.

        \item[{Case }2.] If $\card{V} = 2n - 1$ or $\card{V} = 2n$, it follows from Lemma
            \ref{lem:lower_erdos} that \mvninline{2n-1}{n} = \mvninline{2n}{n} =
            $\binom{2n - 1}{n}$. Since $\card{E} < r \leq \binom{2n - 1}{n}$,
            $H$ is properly \twocolor. 

        \item[{Case }3.] If $2n + 1 \leq \card{V} \leq v_{\text{opt}}$, consider a balanced
            coloring of $H$. We say that such a coloring is \itt{blocked} by a
            \he\ if it is \mono\ in the coloring. Note that a red \monohe\
            blocks $\tbinom{ \card{V}- n}{\lfloor \card{V}/2 \rfloor - n}$ and a
            blue \monohe\ blocks $\binom{\card{V} - n}{\ceil{\card{V}/2} - n}$
            such colorings.  In order to ensure that none of these balanced
            colorings is a proper \twocoloring\ of $H$, we need at least
            $f(\card{V})$ \hes.  Since $\card{E} < r \leq f(\card{V})$ for $2n +
            1 \leq \card{V} \leq v_{\text{opt}}$, at least one of the balanced
            colorings of $H$ is a proper \twocoloring\ of it.

        \item[{Case }4.] If $\card{V} > v_{\text{opt}}$, assume the induction hypothesis
            that any \nuniform\ \hg\ with $\card{V} - 1$  vertices and
            $\card{E}$ \hes\ is properly \twocolor. The base case $\card{V} =
            v_{\text{opt}}$ is proved in Case 3. If there exists a pair of
            vertices $\{v_i, v_j\}$ not contained together in any hyperedge of
            $H$, consider a new \hg\ \hgprimenot\ constructed by merging $v_i$
            and $v_j$ into a new vertex $v$. Since $H'$ is \nuniform\ with
            $\card{V'} = \card{V} - 1$ and $\card{E'} = \card{E}$, we know from
            the induction hypothesis that $H'$ is properly \twocolor. This
            coloring of $H'$ can be extended to a proper \twocoloring\ of $H$ by
            assigning the color of $v$ to $v_i$ and $v_j$.  Since $\card{E} < r$
            and it follows from Lemma \ref{lem:lower_schonheim} that the minimum
            number of \hes\ required to ensure that each pair of vertices is
            contained in at least one \he\ is $g(\card{V}) \geq r$, we are
            guaranteed to have a pair of vertices $\{v_i, v_j\}$ not
            contained together in any \he\ of $H$. 
    \end{enumerate}
\end{proof}

Lemma \ref{lem:goldberg} implies that $m(5) \geq 28$, which is obtained when 
$x=23$. We improve this to $m(5) \geq 29$ using the following lemma. The 
first three cases of the proof for this improved lower bound are the same 
as the ones used in the proof above. We use Lemma \ref{lem:radhakrishnan_srini} 
to improve Case 4 of the proof.

\begin{lemma}\cite{radhakrishnan2000improved} 
    Consider a positive integer $\gamma$ and a fraction $p \in [0, 1]$.  Any
    \nuniform\ \hg\ \hgnot\ satisfying $\card{\{ \{e_1, e_2\} : e_1, e_2 \in E,\
        \card{e_1 \cap e_2} = 1 \}} \leq \gamma$ is properly \twocolor\ if $2^{-n+1} (1 -
    p)^n \card{E} + 4 \gamma \big( 2^{-2n+1} p \int_0^1 (1 - (xp)^2)^{n-1}
    \mathrm{d}x \big) < 1$. \label{lem:radhakrishnan_srini}
\end{lemma}

\begin{proof}[Proof of Result \ref{result-improv-lower}]
    Let us consider a \nuniform[5]\ \hg\ \hgnot\ with at most 28 \hes.  We show
    that it is properly \twocolor. 

    \begin{enumerate}
        \item[{Case }1.] If $5 \leq \card{V} \leq 8$, any balanced coloring of its vertex
            set is a proper \twocoloring\ of $H$. 

        \item[{Case }2.] If $\card{V} = 9$ or $\card{V} = 10$, it follows from Lemma
            \ref{lem:lower_erdos} that \mvninline{9}{5} = \mvninline{10}{5} =
            126. Since $\card{E} \leq 28$, $H$ has a proper \twocoloring.

        \item[{Case }3.] If $11 \leq \card{V} \leq 22$, consider a balanced coloring of
            $H$. We observe that a red \monohe\ blocks $\tbinom{ \card{V}-
                5}{\lfloor \card{V}/2 \rfloor - 5}$ and a blue \monohe\ blocks
            $\binom{\card{V} - 5}{\ceil{\card{V}/2} - 5}$ such colorings.  In
            order to ensure that none of these balanced colorings is a proper
            \twocoloring\ of $H$, we need at least $\mnlowerblocked{\card{V}}
            {5}$ \hes. Since $11 \leq \card{V} \leq 22$, it implies that we need
            at least 29 \hes\ to ensure that no balanced coloring of $H$ is a
            proper \twocoloring.

        \item[{Case }4.] If $\card{V} = 23$ and there exists a pair of vertices $\{v_i,
            v_j\}$ not contained together in any \he\ of $H$, we construct a new
            \hg\ \hgprimenot\ by merging vertices $v_i$ and $v_j$ into a new
            vertex $v$. We observe that $H'$ is \nuniform[5]\ with 22 vertices
            and \cardinline{E} \hes. It follows from Case 3 that $H'$ is
            properly \twocolor.  This coloring of $H'$ can be extended to a
            proper \twocoloring\ of $H$ by assigning the color of $v$ to $v_i$
            and $v_j$. If $\card{E} \leq 27$, note that Lemma
            \ref{lem:lower_schonheim} ensures that there exists a pair of
            vertices not contained together in any \he\ of $H$.  Therefore, we
            would complete the proof by assuming that $\card{E} = 28$ and  every
            pair of vertices is contained in at least one \he\ of $H$.  For such a
            \hg, we show that the cardinality of the set  $\{ \{e_1, e_2\} :
            e_1, e_2 \in E,\ \card{e_1 \cap e_2} = 1 \}$ is at most 335. Setting
            $p = 0.3, \gamma = 335, n = 5$ and $\card{E} = 28$ in Lemma
            \ref{lem:radhakrishnan_srini}, we observe that $H$ is properly
            \twocolor\ since $2^{-n+1} (1 - p)^n \card{E} +  4 \gamma \cdot
            2^{-2n+1} p \int_0^1 (1 - (xp)^2)^{n-1} \mathrm{d}x < 1$. 

            In order to show that the cardinality of the set $\{ \{e_1, e_2\} :
            e_1, e_2 \in E,\ \card{e_1 \cap e_2} = 1 \}$ is at most 335, we
            consider the degree sequence of $H$. Note that the \itt{degree} of a
            vertex is defined as the number of \hes\ it is contained in and the
            \itt{degree sequence} of a \hg\ is the ordering of the degrees of
            its vertices in a non-increasing order.  Consider an arbitrary
            vertex $u$ of $H$.  Observe that  there are 22 distinct vertex pairs
            involving $u$ and any \he\ containing $u$ has 4 such pairs in it.
            Therefore, the degree of $u$ is at least 6 and there exists another
            vertex $u'$ such that $\{u, u'\}$ is contained in at least two
            different \hes\ of $H$.  Since the sum of the degrees of the
            vertices of $H$ is 140, the only  possible degree sequences of $H$
            are  $\langle 8, 6, \ldots, 6 \rangle$ and $\langle 7, 7, 6, \ldots, 
            6 \rangle$. For the first sequence, the cardinality of the set $\{
            \{e_1, e_2\} : e_1,e_2 \in E,\ \card{e_1 \cap e_2} = 1 \}$ is upper
            bounded by $(\binom{6}{2} - 1) \cdot 22 + (\binom{8}{2} - 1) = 335$.
            For the second sequence, it is upper bounded by $(\binom{6}{2} - 1)
            \cdot 21 + (\binom{7}{2} - 1) \cdot 2 = 334$. 

        \item[{Case }5.] If $\card{V} \geq 24$, assume the induction hypothesis that any
            \nuniform[5]\ \hg\ with $\card{V} - 1$  vertices and $\card{E}$
            \hes\ is properly \twocolor. The base case $\card{V} = 23$ is proved
            in Case 4. If there exists a pair of vertices $\{v_i, v_j\}$ 
            not contained together in any \he\ of $H$, consider a new \hg\
            \hgprimenot\ constructed by merging $v_i$ and $v_j$ into a new
            vertex $v$. Since $H'$ is \nuniform[5] with $\card{V'} = \card{V}
            - 1$ and $\card{E'} = \card{E}$, we know from the induction
            hypothesis that $H'$ is properly \twocolor. This coloring of $H'$
            can be extended to a proper \twocoloring\ of $H$ by assigning the
            color of $v$ to $v_i$ and $v_j$.  Since it follows from Lemma
            \ref{lem:lower_schonheim} that the minimum number of \hes\ required
            to ensure that each pair of vertices is contained in at least one
            \he\ is $\mnlowerschoheim{\card{V}}{5}{4} \geq 29$, we are guaranteed
            to have a pair of vertices $\{v_i, v_j\}$ not contained together in
            any \he\ of $H$. 
    \end{enumerate}
\end{proof}

\section{Conclusion \label{concl-open}}
In this paper, we establish the lower bound $m(5) \geq 29$ which is
still far from the \bk\ upper bound $m(5) \leq 51$.  We also establish
improved upper bounds for \mninline[8], \mninline[13], \mninline[14],
\mninline[16]\ and \mninline[17].  In Table \ref{table-new-known-bounds}, we
highlight these improved bounds on \mninline\ for $n \leq 17$. It would be
interesting to determine the exact values of \mninline\ for $n \geq 5$. 

\begin{table}
\begin{center}
\begin{tabular}{|c|c|c|}
\hline
$n$ & $m(n)$ & Corresponding construction/recurrence relation\\
\hline
1   & $m(1) = 1$            & Single Vertex\\
2   & $m(2) = 3$            & Triangle Graph\\
3   & $m(3) = 7$            & Fano Plane \cite{klein1870theorie}\\
4   & $m(4) = 23$           & \cite{ostergaard2014minimum},
\cite{seymour1974note}\\
5   & $m(5) \leq 51$        & $m(5) \leq 2^4 + 5 m(3)$\\
6   & $m(6) \leq 147$       & $m(6) \leq m(2) m(3)^2$\\
7   & $m(7) \leq 421$       & $m(7) \leq 2^6 + 7 m(5)$\\
8   & $m(8) \leq {\bf \meightmultiimprov} $      & $m(8) \leq 2 m(3) m(5) + 
                                    \tbinom{4}{2} m(3) m(3) + \tbinom{4}{3} m(5) $\\
9   & $m(9) \leq 2401$      & $m(9) \leq m(3)^4$\\
10  & $m(10) \leq 7803$     & $m(10) \leq m(2) m(5)^2$\\
11  & $m(11) \leq 25449$    & $m(11) \leq 15 \cdot 2^8 + 9 m(9)$\\
12  & $m(12) \leq 55223$    & $m(12) \leq m(3)^4 m(4)$\\
13  & $m(13) \leq {\bf \mthirteenblockimprov}$   & $m(13) \leq (m(3) +
                                    2^3)m(5)^2 +  2m_{H_f}(5)m(4)^2 + 4m(5)m(4)^2$ \\
14  & $m(14) \leq {\bf \mforteenmultiimprov}$   & $m(14) \leq 2 m(9) m(5) +
                                    m(5)^2 m(4) + \tbinom{4}{1} m(9) m(4) + 
                                    \tbinom{4}{4} m(5)^2$\\
15  & $m(15) \leq 857157$   & $m(15) \leq m(3)^5 m(5)$\\
16  & $m(16) \leq {\bf \msixteenblockimprov}$  & $m(16) \leq (m(4) +
                                    2^{4})m(6)^2 + 2m_{H_s}(6)m(5)^2 + 4m(6)m(5)^2$\\
17  & $m(17) \leq {\bf \mseventeenmultiimprov}$ & $m(17) \leq 2 m(10) m(7) + 2
                                    m(7)^2 m(3) + \tbinom{3}{1}m(10) m(3)^2 + 
                                    \tbinom{3}{3} m(7)^2 $ \\
\hline
\end{tabular}
\end{center}
\caption{Improved upper bounds on $m(n)$ for small values of $n$}
\label{table-new-known-bounds}
\end{table}


\small
\begin{thebibliography}{2200}
 \bibliographystyle{plain}

\bibitem{abbott1969combinatorial}
H. L. Abbott and D. Hanson. On a combinatorial problem of Erd\H{o}s.
{Canadian Mathematical Bulletin},
12 (6), 823-829 (1969).

\bibitem{abbott1964combinatorial}
H. L. Abbott and L. Moser. On a combinatorial problem of Erd\H{o}s and Hajnal.
{Canadian Mathematical Bulletin},
7, 177-181 (1964).

\bibitem{beck19783}
J. Beck. On 3-chromatic hypergraphs.
{Discrete Mathematics},
24 (2), 127-137 (1978).

\bibitem{cherkashin2014note}
D. D. Cherkashin and J. Kozik. A note on random greedy coloring of uniform hypergraphs.
{Random Structures and Algorithms},
47 (3), 407-413 (2015).

\bibitem{colbourn2006handbook}
C. J. Colbourn and H. J. Dinitz. Handbook of Combinatorial Designs.
{CRC press} (2006).

\bibitem{erdos1963combinatorial}
P. Erd\H{o}s. On a combinatorial problem.
{Nordisk Mat. Tidskrift},
11, 5-10 (1963).

\bibitem{erdos1969combinatorial}
P. Erd\H{o}s. On a combinatorial problem III.
{Canad. Math. Bull.},
12 (4), 413-416 (1969).

\bibitem{gebauer2013construction}
H. Gebauer. On the construction of 3-chromatic hypergraphs with few edges.
{Journal of Combinatorial Theory, Series A},
120 (7), 1483-1490 (2013).

\bibitem{Goldberg93towardcomputing}
M. K. Goldberg and H. C. Russell. Toward computing m(4).
{Ars Combinatoria},
3, 139-148 (1993).

\bibitem{klein1870theorie}
F. Klein. Zur theorie der liniencomplexe des ersten und zweiten grades.
{Mathematische Annalen},
2 (2), 198-226 (1870).

\bibitem{mathews2015construction}
J. Mathews, M. K. Panda and S. Shannigrahi. On the construction of non-2-colorable uniform hypergraphs.
{Discrete Applied Mathematics},
180, 181-187 (2015).

\bibitem{ostergaard2014minimum}
P. R. J. {\"O}sterg{\aa}rd. On the minimum size of 4-uniform hypergraphs without property B.
{Discrete Applied Mathematics},
163 (2), 199-204 (2014).

\bibitem{radhakrishnan2000improved}
J. Radhakrishnan and A. Srinivasan. Improved bounds and algorithms for hypergraph 2-coloring.
{Random Structures and Algorithms},
16 (1), 4-32 (2000).

\bibitem{seymour1974note}
P. D. Seymour. A note on a combinatorial problem of Erd\H{o}s and Hajnal.
{Journal of the London Mathematical Society},
s2-8 (4), 681-682 (1974).

\bibitem{toft1973color}
B. Toft. On colour-critical hypergraphs.
{Colloq. Math. Soc. J\'anos Bolyai 10, Infinite and Finite Sets},
Volume III, 1445-1457 (1975).
\end{thebibliography}
\end{document}